\documentclass[11pt]{article}
\usepackage[T1]{fontenc}
\usepackage{latexsym,amssymb,amsmath,amsfonts,amsthm}
\usepackage{graphics}
\usepackage{graphicx}
\usepackage{mathrsfs}
\usepackage{subfigure}
\usepackage{amsmath}
\usepackage{algorithm}
\usepackage{algorithmic}
\usepackage[numbers,sort&compress]{natbib}
\usepackage{color}
\newcommand{\g}{\gamma}
\newcommand{\G}{\Gamma}
\newcommand{\eps}{\epsilon}
\newcommand{\vep}{\varepsilon}
\newcommand{\Om}{\Omega}
\newcommand{\om}{\omega}

\newcommand{\ka}{\kappa}
\newcommand{\al}{\alpha}
\newcommand{\la}{\lambda}

\newcommand{\wi}{\widetilde}

\newcommand{\Int}{\int\limits}

\newcommand{\pa}{\partial}

\newcommand{\ov}{\overline}

\newcommand{\dive}{{\rm div\,}}

\newcommand{\ti}{\times}
\newcommand{\wid}{\widetilde}
\newcommand{\de}{\delta}

\newcommand{\na}{\nabla}
\newcommand{\mat}{\mathbb}
\newcommand{\se}{\setminus}
\newcommand{\ify}{\infty}

\newcommand{\tr}{\triangle}

\newcommand{\R}{{\mat R}}

\newcommand{\C}{{\mat C}}

\newcommand{\no}{\nonumber}
\newcommand{\be}{\begin{eqnarray}}
\newcommand{\ben}{\begin{eqnarray*}}
\newcommand{\en}{\end{eqnarray}}
\newcommand{\enn}{\end{eqnarray*}}
\newtheorem{theorem}{Theorem}[section]
\newtheorem{lemma}[theorem]{Lemma}

\newtheorem{remark}[theorem]{Remark}

\makeatletter
\newenvironment{tablehere}
  {\def\@captype{table}}
  {}
\newenvironment{figurehere}
  {\def\@captype{figure}}
  {}
\makeatother

\begin{document}
\renewcommand{\theequation}{\arabic{section}.\arabic{equation}}

\begin{titlepage}
\title{\bf 
A non-iterative sampling method for inverse elastic wave scattering by rough surfaces}
\author{
Tielei Zhu\thanks{School of Mathematics and Statistics, Xi'an Jiaotong University, Xi'an, Shaanxi, 710049, China.
(\tt zhutielei@stu.xjtu.edu.cn)}\;\;,
~~Jiaqing Yang\thanks{School of Mathematics and Statistics, Xi'an Jiaotong University, Xi'an, Shaanxi, 710049, China.({\tt jiaq.yang@xjtu.edu.cn})}
}

\date{}
\end{titlepage}
\maketitle
\begin{abstract}

Consider the two-dimensional inverse elastic wave scattering by an infinite rough surface with a Dirichlet boundary condition. A non-interative sampling technique is proposed for detecting 
the rough surface by taking elastic wave measurements on a bounded line segment above the surface, based on reconstructing a modified near-field equation associated with a special surface, 
which generalized our pervious work for the Helmholtz equation  (SIAM J. IMAGING. SCI.  10(3)(2017), 1579-1602)  to the Navier equation.
Several numerical examples are carried out to illustrate the effectiveness of the inversion algorithm.

\end{abstract}


\section{Introduction}\label{sec1}
\setcounter{equation}{0}

This paper is concerned with the two-dimensional inverse elastic wave scattering by a non-locally rough surface which is assumed to lie in an unbounded strip domain with a finite height in $x_2$-direction.
This kind of problems have received much attention in both engineering and mathematics due to  their wide range of applications in many fields such as geophysics, nondestructive testing and seismology. Precisely, the elastic
surface is assumed to be rigid which means that the elastic wave field satisfies a Dirichlet boundary condition on the surface. Moreover, a suitable radiation condition is needed to described
the asymptotic behavior of the scattered field away from the surface. Then the forward problem is to determine the distribution of the associated scattered field, provided an incident field and the elastic 
surface are given. Existence of a unique solution has been shown in different function space settings for elastic scattering by rough surfaces. We refer the reader to \cite{T01a,T02,JG12,JG15} for a detailed investigation.

In this paper, we are focused on the study of numerical solution of an inverse elastic scattering problem associated with a non-locally rough surface. The purpose is to propose a simple imaging technique to detect the shape and location of the surface from the wave field measurements above the surface. Many inversion algorithms have been proposed in the literature if the rough surface is scattered by an incident acoustic field. For instance, a nonlinear integral equation method was proposed in \cite{Li2015A} for recovering an impenetrable rough surface with a Dirichlet boundary condition, and a direct sampling method was proposed in \cite{liu2018a} for imaging an impenetrable surface or an interface in dielectric media by taking near-field Cauchy data.
If the rough surface is assumed to be a small and smooth deformation of a plane, a transformed field expansion method was proposed in \cite{GP13} for recovering the surface with a Dirichlet, impedance or transmission condition, and a factorization method was proposed in \cite{lechleiter2008factorization} for recovering a Dirichlet rough surface under the assumption $\kappa f_+<\sqrt{2}$, where $\kappa>0$ stands for the wavenumber
and $f_+$ stands for the height of the rough surface. Meanwhile, it is noticed that a time-domain point source method and a probe method were also proposed in \cite{lines2005time,Burkard2017A}, respectively, for recovering a rough surface with a Dirichlet boundary condition. If the rough surface is
considered to be a local perturbation of a planar surface, a Newton iterative method was introduced in \cite{Zhang2013A} for detecting an impenetrable surface, and the Kirsch-Kress method was extended in \cite{li2017kirsch} to recover a penetrable interface by
near-filed measurements above and below the surface. We also refer to \cite{MJKJ17} for an application of the so-called linear sampling method for imaging a Dirichlet rough surface or a penetrable interface. However, 
we remark that it is not trivial to extend the above methods to numerically solve the inverse elastic scattering problems associated with an infinite rough surface, due to the coupling of the compressional and shear waves which brings new 
challenges in both mathematics and numerics. A recent attempt in this direction is due to Liu {\em et. al} \cite{XBH19,hu2019a} in which a direct sampling method was proposed for imaging an elastic surface by using plane waves or point sources, extending their previous works for the Helmholtz equations \cite{liu2018a}. 

In this paper, we aim to study the linear sampling method (LSM) as analytic as a tool to numerically reconstruct an infinite elastic surface. It is well-known that the LSM is an efficient imaging technique and has been extensively studied in inverse scattering problems by bounded obstacles, for example, \cite{cakoni2011linear}, since the reduced algorithm is fast and does not need any a priori knowledge on the obstacles. 
However, it is challenging to present a strict theoretical analysis on the LSM for recovering an infinite surface when compared to the bounded obstacle case. We refer the reader to a recent work \cite{MJKJ17} on a modified version of the LSM for recovering locally rough surfaces or interfaces by incident acoustic waves. Partially motivated by \cite{MJKJ17}, we reformulate the elastic scattering problem into an equivalently boundary value problem (BVP) with compactly supported boundary data, by introducing a class of specially rough surfaces. The well-posedness of the BVP directly follows from the classical Riesz-Fredholm alternative. A modified near-field equation is then proposed to numerically reconstruct the elastic surface, where the surface can be reconstructed more fully as the auxiliary parameter $R>0$ is chosen 
to be sufficiently large in the modified near-field equation.

The reminder of this paper is organized as follows. In section 2, we propose the mathematical formulation for the elastic surface scattering by point sources, which is reduced to 
a BVP with compactly support boundary data by introducing a Dirichlet-Green function for a special rough surface. Some necessary properties
on the solution of the BVP are then shown. In section 3, a rigorous theoretical analysis is provided for recovering a non-locally rough surface by the extended LSM based on a modified near-field equation.
In section 4, we carry out several numerical examples to demonstrate the effectiveness of the reduced algorithm, and present a brief conclusion in Section 5.



\section{Mathematical formulation}\label{sec2}
\setcounter{equation}{0}

In this section, we shall present the formulation of the model problem for the two-dimensional elastic wave scattering by an unbounded rough surface with an incident point source; see Figure \ref{fig:label1} for a geometric setting.
\begin{figure}[ht]
\centering
\includegraphics[scale=0.25]{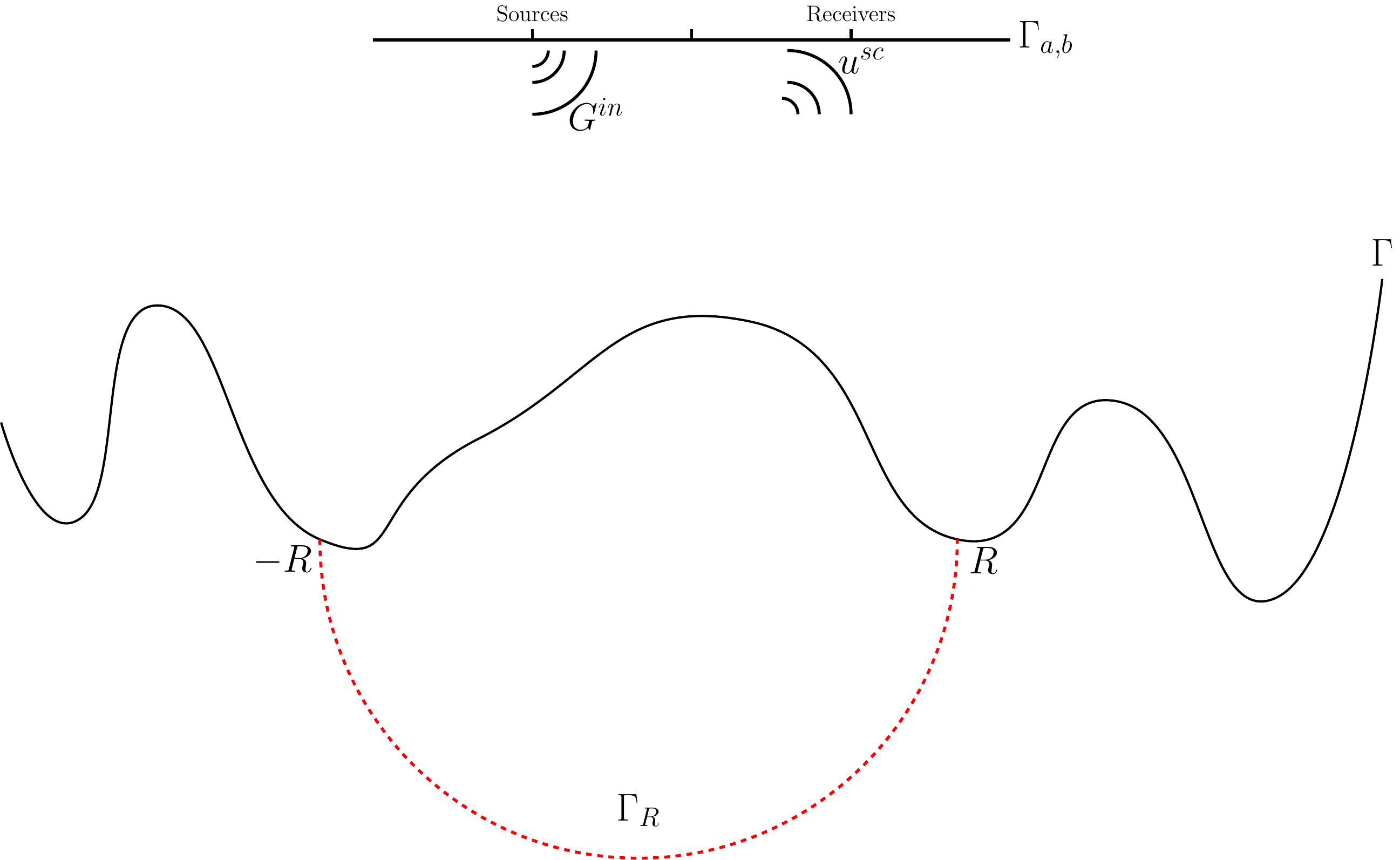}
\caption{Elastic wave scattering by unbounded rough surfaces}
\label{fig:label1}
\end{figure}
More precisely, let the surface be denoted by $\Gamma:=\{(x_1,x_2)\in\R^2:x_2=f(x_1),\;x_1\in\R\}$, where $f$ is assumed to be a Lipschitz continuous function satisfying that there exists two constants $f_-$ and $f_+$ such that $f_-<f(x)<f_+$ for all $x\in\Gamma$. Define the unbounded domain above $\G$ in $\R^2$ by $\Omega:=\{x=(x_1,x_2)\in \R^2:x_2>f(x_1),\;x_1\in\R\}$ which is
filled with isotropic homogeneous medium described by lame coefficients $\la$, $\mu$ with $\mu>0$, $\la+\mu>0$.
For $s\in \R$, we define the domain $U_s:=\{(x_1,x_2)\in\R^2:x_2>s \}$ and $\Gamma_s:=\{(x_1,x_2)\in\R^2:x_2=s\}$. Then it holds $U_{f_+}\subset \Om\subset U_{f_-}$. 
Without loss of generality, we assume throughout the paper  $f_-<0<f_+$ and set $S_h:=\Om\se\ov{U}_h$ for $h>f_+$.


Consider the incident wave $G^{in}(x,y,p)$ induced by an elastic point source in the form 
\be\label{2.1a}
G^{in}(x,y,p): = \Gamma(x,y)p,\qquad x,y\in\R^2,\;x\neq y,
\en
where $p\in\C^2$ is the polarization vector and $\Gamma(x,y)$ is given by
\ben
\G(x,y):=\frac{\rm i}{4\mu} H_0^{(1)}(k_s|x-y|){\bf I}+
\frac{\rm i}{4\om^2}\na_{x}^{\perp}\na_{x}(H_0^{(1)}(k_s|x-y|)-H_0^{(1)}(k_p|x-y|))
\enn
which corresponds to the fundamental solution to the Navier equation in $\R^2$. Here, 
$H^{(1)}_0(\cdot)$ is the Hankel function of first kind of order zero, and 
\ben
 k_s:=\frac{\om}{\sqrt{\mu}},\qquad k_p:=\frac{\om}{\sqrt{2\mu+\la}}
\enn
for the frequency $\omega\in\R$.

If $\Gamma$ is considered to be a rigid surface in elasticity, then the scattering of $G^{in}$ by $\Gamma$ can be formulated as finding a scattered field $u^{sc}$ so that $u^{sc}$ satisfies the Navier equation
\be\label{2.2a}
&&\tr^* u^{sc}+\om^2 u^{sc} = 0 \quad {\rm in\;}\Om,\qquad \tr^*:=\mu\tr+(\la+\mu){\rm grad\,}\dive,\\\label{2.2}
&&u^{sc}(x) = -G^{in}(x,y,p)\quad {\rm on\;}\Gamma,
\en
and an appropriate radiation condition. To formulate the radiation condition for the model, we introduce a modified incident wave
\be\label{2.3a}
\wi{G}^{in}(x,y,p)=(1-\phi_y(x) )G^{in}(x,y,p)
\en
where $\phi_y(x)\in C^{\infty}(\Om)$ is defined by $\phi_y(x)=1$ for $|x-y|\leq \frac{1}{2}\vep$ and $\phi_y(x)=0$ for $|x-y|\geq\vep$ with $\vep<dist(y,\G)$.

Notice that the scattering problem (\ref{2.1a})-(\ref{2.3a}) can be viewed as a boundary value problem with the boundary data on $\Gamma$ induced by (\ref{2.1a}). This means that 
the scattered solution $u^{sc}$  to (\ref{2.1a})-(\ref{2.3a}) is not changed if the modified incident wave (\ref{2.3a}) is used instead of (\ref{2.1a}). Then we can define the modified total field
\be\label{2.4a}
\wi{ u}(x,y,p)=u^{sc}(x,y,p)+\wi{G}^{in}(x,y,p)\qquad {\rm in\;}\R^2.
\en
It is easily checked that $\wi{ u}$ satisfies the following problem:
\be\label{2.6a}
&&\tr^* \wi{u}+\om^2 \wi{u}= g_y \quad {\rm in\;}\Om,\\\label{2.5a}
&&\qquad\qquad\;\wi{u} = 0 \quad\ {\rm on\ }\G,
\en
where $g_y(x):=(\tr^*_x+\om^2 )\wi{G}^{in}(x,y,p)\in L^2(\R^2)$. Moreover, it follows from \cite{T01a,JG12} that $\wi{u}$
satisfies the so-called upward propagating radiation condition (UPRC)
\be\label{2.7a}
\wi{u}(x)=\frac{1}{\sqrt{2\pi}}\int_{\R}
\left(e^{{\rm i}\g_p(\xi)(x_2-h)}M_p(\xi)+e^{{\rm i}\g_s(\xi)(x_2-h)}M_s(\xi)\right)\hat{\wi{u}}_h(\xi)e^{ix_1\xi}{\,\rm d}\xi
\en
for $x_2>h$ with $supp\, g\subset S_h$ and $h>f_+$, where $M_p$ and $M_s$ are given by

%
%
%
\ben
M_p(\xi)=\frac{1}{\xi^2+\g_p\g_s}
\begin{pmatrix}
  \xi^2 &\xi\g_s\\
  \xi\g_p &\g_p\g_s
\end{pmatrix}
,\quad
M_s(\xi)=\frac{1}{\xi^2+\g_p\g_s}
\begin{pmatrix}
   \g_p\g_s &-\xi\g_s\\
  -\xi\g_p & \xi^2
\end{pmatrix},
\enn
respectively, with
\ben
\g_s(\xi)=\sqrt{k_s^2-\xi^2},\qquad \g_p(\xi)=\sqrt{k_p^2-\xi^2},
\enn 
and $\hat{\wi{u}}_h$ the Fourier transform of $\wi{u}_h:=\wi{u}|_{\G_h}$.

For $H>f_+$, let $V_H$ as the closure of $ [C^{\ify}_0(S_H\bigcup\G_H)]^2$ under the norm
\ben
\|u\|_{V_H}=\left(\|\na u\|^2_{[L^2(S_H)]^{2\ti 2}} +\|u\|^2_{[L^2(S_H)]^2}\right)^{\frac{1}{2}}.
\enn
Then Problem (\ref{2.3a})-(\ref{2.7a}) can be reformulated into the following boundary value problem:{\em 
find $\wi{u}\in [H^1_{loc}(\Om)]^2$ such that  $\wi{u}|_{S_H}\in V_H$ for every  $H>h$ which satisfies (\ref{2.6a}) in a distributional sense and (\ref{2.7a}) with $\wi{u}_h\in [H^{\frac{1}{2}}(\G_h)]^2$.}



Next, we briefly recall the well-posedness to Problem (\ref{2.3a})-(\ref{2.7a}) by a variational method (cf.\cite{JG12}). For $\wid{\mu}\in\R$ and $\wid{\lambda}\in\R$, 
we first introduce the first generalized Betti formula
\ben
-\int_{D}(\tr^*+\om^2)w\cdot  v {\,\rm d}x
=\int_{D}( \mathcal{E}_{\wid{\mu},\wid{\la}}(w,v)-\om^2w\cdot v){\,\rm d}x - \int_{\pa D} v\cdot T_{\wid{\mu},\wid{\la}}w {\,\rm d}s
\enn
and
the third generalized Betti formula
\ben
\int_{D}\tr^*w\cdot v-\tr^*v\cdot w {\,\rm d}x
= \int_{\pa D}(v\cdot T_{\wid{\mu},\wid{\la}}w -w\cdot T_{\wid{\mu},\wid{\la}}v ) {\,\rm d}s
\enn
associated with the differential operator $\tr^*$ for $w,v\in [H^2(D)]^2$, where
\ben
\mathcal{E}_{\wid{\mu},\wid{\la}}(w,v)&:=&(\la+2\mu)(\pa_1w_1\pa_1v_1+\pa_2w_2\pa_2v_2)+\mu(\pa_2w_1\pa_2v_1+\pa_1w_2\pa_1v_2)\\
 & & +\wid{\la}(\pa_1w_1\pa_1v_1+\pa_2w_2\pa_2v_2)+\wid{\mu}\pa_2w_1\pa_1v_2+\pa_1w_2\pa_2v_1),\\
T_{\wid{\mu},\wid{\la}}w&:=&(\mu+\wid{\mu})\pa_{\bf n}w+\wid{\la}{\bf n}\dive w+\wid{\mu}
\begin{pmatrix}
  n_2(\pa_1w_2-\pa_2w_1)\\
  n_1(\pa_2w_1-\pa_1w_2)
\end{pmatrix}
\enn
with the unit exterior normal ${\bf n}: = (n_1,n_2)\in \mat{S}^1$ on $\partial D$, 
 and $D$ denotes a domain in which the divergence theorem is assumed to hold. For convenience, we have used the notation $\partial_j$ to indicate $\partial_j: = \partial/\partial_{x_j}$
 for $j=1,2$. 
 
Define 
\be\label{2.8a}
\begin{aligned}
\mathcal{E}(w,v): &=& \mathcal{E}_{0,\la+\mu}(w,v),\\
Tw: &=& T_{0,\la+\mu}w.
\end{aligned}
\en
In order to truncate Problem (\ref{2.3a})-(\ref{2.7a}) into a bounded strip in $x_2$-direction, we introduce the Dirichlet-to-Neumann map $\mathcal{T}$ (cf.\cite{JG12}) on $\G_h$ by 
\be\label{2.9a}
\mathcal{T}v:=\mathcal{F}^{-1}(M(\xi)\mathcal{F}v),\qquad {\rm for\;}v\in [C^{\ify}_0(\G_h)]^2,
\en
where 
the matrix $M(\xi)$ is given by
\ben
M(\xi)=\frac{{\rm i} }{\xi^2+\g_p\g_s}
\begin{pmatrix}
  \om^2\g_p & -\xi\om^2+\xi\mu(\xi^2+\g_p\g_s)\\
  \xi\om^2-\xi\mu(\xi^2+\g_p\g_s) & \om^2\g_s
\end{pmatrix},
\enn
and $\mathcal{F}$ and $\mathcal{F}^{-1}$ denote the Fourier transform and the inverse Fourier transform, respectively.
It follows from \cite{JG12} that $\mathcal{T}$ can extended to a linear bounded map from $[{H}^{\frac{1}{2}}(\G_h)]^2$ to $[{H}^{-\frac{1}{2}}(\G_h)]^2$.

Using the Betti formulas, Problem (\ref{2.3a})-(\ref{2.7a}) can be reformulated in the variational form as finding 
$\wi{u}\in V_h$ such that
\be\label{2.10a}
\begin{aligned}
B(\wi{u},v): &= \int_{S_h}(\mathcal{E}(\wi{u},\ov{v})-\om^2u\cdot\ov{v} ){\,\rm d}x-\Int_{\G_h}\g_-\ov{v}\cdot \mathcal{T}\g_-\wi{u} {\,\rm d}s\\
            &= \int_{S_h}\ov{v}\cdot g_y {\,\rm d}x \\
            &= :F(v)\qquad\qquad {\rm for\;all\;}v\in V_h,
\end{aligned}
\en
where
$\g_-$ denotes the trace restricted to $\G_h$. By the Risze representation, the sesquilinear $B(\cdot,\cdot)$ induces a bounded linear operator 
$\mathcal{B}(\om): V_h\to V_h^*$ for a fixed $\om>0$ such that
\ben
B(w,v)=(\mathcal{B}(\om)w,v)\qquad {\rm for\;all\;}w,v\in V_h.
\enn

\begin{theorem}\label{thm2.1} There exists a unique solution to Problem (\ref{2.6a})-({\ref{2.7a}}) for any given $g_y\in L^2(\R^2)$.
\end{theorem}

\begin{proof}
The assertion follows directly from the isomorphism of the operator $\mathcal{B}(\om)$ from $V_h$ to $V_h^*$ for any $\om>0$. We refer to \cite{JG12} for a detailed discussion.
\end{proof}

In the following, we will introduce an auxiliary boundary value problem which is related to a special rough surface 
\be\label{2.12aa}
\wid{\G}_R: = \pa\Om_R,
\en
where $\Om_R:=\Omega \cup B_R$ is a perturbed domain of $\Om$ and $B_R:=\{x\in \R^2:|x|<R \}$  with $R>0$ sufficiently large so that $\Gamma\cap B_R$ is a connected open segment of $\Gamma$. It is easily observed that $\wid{\G}_R$ is just a local perturbation of the rough surface $\Gamma$. That is, $\wid{\G}_R$ is different from $\Gamma$ on a finite interval $[-R,R]$.

Consider the scattering of the same point source $G^{in}(\cdot,y,p)$ by the perturbed surface $\wid{\G}_R$. Then one has the following problem
\be\label{2.11a}
\begin{aligned}
  &\tr^*G^{sc}(\cdot,y,p;R)+\om^2G^{sc}(\cdot,y,p;R)=0\qquad  &{\rm in}\ \Om_R,\\
  &G^{sc}(\cdot,y,p;R)+G^{in}(\cdot,y,p)=0\qquad & {\rm on} \ \wid{\G}_R,\\
 &  \wi{G}:=G^{sc}+\wi{G}^{in}  {\ \rm satisfies \ the\ UPRC}.  &
\end{aligned}
\en
It follows from Theorem \ref{thm2.1} that Problem (\ref{2.11a}) is uniquely solvable for any given incident wave $G^{in}(\cdot,y,p)$ with $y\in\Om$ and $p\in\C^2$.

Let $u(x,y,p):=u^{sc}(x,y,p)+G^{in}(x,y,p)$ and $G(x,y,p;R):=G^{sc}(x,y,p;R)+G^{in}(x,y,p)$ denote the total fields for the rough surfaces $\Gamma$ and $\wid{\G}_R$, respectively.
Define the difference between $u^{sc}(x,y,p)$ and $G^{sc}(x,y,p;R)$ by $v(\cdot,y,p;R):=u(\cdot,y,p)-G(\cdot,y,p;R)$ in $\Omega$. It is checked that $v(\cdot,y,p;R)$ satisfies the Navier equation 
\be\label{2.12a}
(\tr^*+\om^2)v(\cdot,y,p;R)=0\qquad {\rm in }\ \Om
\en
and the UPRC.  Moreover, it can be also found by the boundary conditions for $u$ and $G$ that 
$v$ has compactly support on $\G$, i.e.,
\be\label{2.13a}
v(x,y,p;R)=0\qquad {\rm for }\ x\in \G \ {\rm with}\ |x_1|>R.
\en
Based on the above observations, we introduce the following boundary value problem
\be\label{2.14a}
\begin{aligned}
  &(\tr^*+\om^2)v =0\qquad  &{\rm in}\ \Om,\\
  &v=\varphi\qquad & {\rm on} \ \G,\\
 & v {\ \rm satisfies \ the\ UPRC}, &
\end{aligned}
\en
where $\varphi$ is the boundary data with compact support on $\G$. Notice that it is possible to recover the surface segment $\G\se\wid{\G}_R$ by knowledge of the difference $v(\cdot,y,p;R)$
for the inverse problem. Next we shall obtain some important properties on the solution of Problem (\ref{2.14a}) which is necessary to justify the linear sampling method in the unbounded case.

To analyse Problem (\ref{2.14a}), we first introduce some useful notations and function spaces, which will be frequently used later. For $R>0$, definite two subsets $\G^{(0)}_R$ and  $\G^{(1)}_R$
of $\G$ by
\ben
\G_{R}^{(0)}: = \{x\in\G: |x_1|\geq R\}\qquad{\rm and}\qquad
\G_{R}^{(1)}: = \G\se \G_{R}^{(0)},
\enn
respectively. For the index $s\in\R$, define the function spaces
\be\label{2.15a}
[H^s(\G_{R}^{(1)})]^2:&=&\{\psi|_{\G_{R}^{(1)}}:  \psi\in [H^s(\G)]^2 \}\quad
    \text{for\;\;}|s|\leq1,\\ \label{2.16a}
[\wid{H}^s(\G_{R}^{(1)})]^2:&=&\{\psi|_{\G_{R}^{(1)}}:
    \psi\in [H^s(\G)]^2,\;\;{\rm supp}(\psi)\subset\ov{\G_{R}^{(1)}}\}\quad \text{for\;\;}|s|\leq1.
\en
It is known that $\langle H^s(\G_{R}^{(1)}),\wid{H}^{-s}(\G_{R}^{(1)})\rangle$ is a pair of dual product for $|s|\leq1$.
Using the Green's theorem again, Problem (\ref{2.14a}) can be rewritten in the variational form of finding $v\in [H^1(S_h)]^2$ with $v|_{\Gamma}=\varphi\in [\wid{H}^{\frac{1}{2}}(\G_{R}^{(1)})]^2$ such that 
\be\label{2.17a}
B(v,\xi)=0,\qquad {\rm for\;all\;\;}\xi\in V_h.
\en
By the property of $\varphi$, we now construct one associated function $v_0\in [H^1(D_{R,h})]^2$ to satisfy the Dirichlet problem 
\be\label{2.18a}
\begin{aligned}
\tr^*v_0=0\quad  {\rm in}\ D_{R,h},\qquad v_0=\wid{\varphi} \quad {\rm on}\ \pa D_{R,h},
\end{aligned}
\en
where $D_{R,h}$ is chosen to be a bounded Lipschitz domain which is contained in $S_h$ with $\ov{\G_R^{(1)}}=\pa D_{R,h} \cap \G$, and
$\wid{\varphi}$ is the zero extension of $\varphi$ from $\G_R^{(1)}$ to $\pa D_{R,h}$. It is known that such $v_0$ is well-defined by the well-posedness of Problem (\ref{2.18a})
with the estimate 
\be\label{2.19a}
\|v_0\|_{H^1(D_{R,h})} \leq C\| \wid{\varphi} \|_{[H^{1/2}(\partial D_{R,h})]^2} = C\|\varphi\|_{[\wid{H}^{1/2}(\G_{R}^{(1)})]^2}.
\en

Define $\wid{v}: = v-v_0$. It is found that Problem (\ref{2.14a}) can be reduced to find $\wid{v}\in V_h$ satisfying 
\be\label{2.20a}
B(\wid{v},\xi)=-B(v_0,\xi)\qquad {\rm for\;all\;\;}\xi\in V_h.
\en
By combining Theorem \ref{thm2.1} with (\ref{2.20a}), we have the following well-posedness on Problem (\ref{2.14a}).
\begin{theorem}\label{thm2.2}
For $\varphi\in [\wid{H}^{1/2}(\G_{R}^{(1)})]^2$, there exists a unique solution $v\in [H^1(S_h)]^2$ with $v|_{\Gamma}=\varphi\in [\wid{H}^{\frac{1}{2}}(\G_{R}^{(1)})]^2$ to
Problem (\ref{2.14a}) such that 
\ben
\|v\|_{[H^1(S_h)]^2}\leq c\|\varphi\|_{[\wid{H}^{1/2}(\G_{R}^{(1)})]^2}.
\enn
\end{theorem}
By Theorem \ref{thm2.2}, we introduce the solution operator $L:[\wid{H}^{\frac{1}{2}}(\G_{R}^{(1)})]^2\to [L^2(\G_{a,b})]^2 $  associated with Problem (\ref{2.14a}), which is defined by
\be\label{2.21a}
L\varphi: = v|_{\G_{a,b}}
\en
where $\G_{a,b}: = \{x\in\R^2: |x_1|\leq a,\; x_2=b\}$ with $b>f_+$, and $v$ is the unique solution of Problem (\ref{2.14a}) with the boundary data $\varphi$. Clearly, $L$ is well defined.
Furthermore, we have the following lemma on $L$.

\begin{lemma}\label{le2.1}
$L$ is injective and has dense range in $[L^2(\G_{a,b})]^2$.
\end{lemma}

\begin{proof}
Let $L\varphi = 0$ for some $\varphi\in [\wid{H}^{\frac{1}{2}}(\G_{R}^{(1)})]^2$. One has $v=0$ on $\G_{a,b}$ from (\ref{2.21a}). Then the uniqueness of the Dirichelt problem in $U_b$ as well as
the unique continuation property yields $v=0$ in $\Omega$, whence $\varphi=0$ follows from the trace theorem.

To show the dense range of $L$ in $[L^2(\G_{a,b})]^2$, it is enough to prove that the adjoint operator $L^*$ of $L$ is injective. We first claim that $L^*$ can be given as 
\be\label{2.22a}
L^*\eta: = T\ov{w}|_{\G_{R}^{(1)}},
\en
where $w$ is the solution to the boundary value problem
\begin{equation}\label{2.23a}
\begin{split}
\tr^* w+\om^2 w = 0 &\quad {\rm in\ }\Om\se\G_{a,b},\\
w = 0 & \quad {\rm on\ }\G,\\
[w]=0,\ \left[Tw\right] = \ov{\eta} & \quad {\rm on\;}\G_{a,b},\\
 w {\ \rm satisfies \ the \ UPRC}
\end{split}
\end{equation}
for $\eta\in [L^2(\G_{a,b})]^2$.  In (\ref{2.23a}), $[\cdot]$ is defined by $[\phi]: = \phi|_+ -\phi|_-$ for
a function $\phi$ with $\cdot|_+$ and $\cdot|_-$ indicating the limits from the top and bottom domains above $\G_{a,b}$, respectively.
Similar to the discussions of Problem (\ref{2.3a})-(\ref{2.7a}) or (\ref{2.14a}), Problem (\ref{2.23a}) can be proved to be well-posed by the variational method
for any given $\eta\in [L^2(\G_{a,b})]^2$. 

 Choose a bounded Lipschitz domain $\Om_{a,b}$ such that $\Om_{a,b}\subset S_H$ for a constant $H>b$ and $\G_{a,b}\subseteq\pa\Om_{a,b}$. Choosing $A>\max\{a,R\}$, we then define a cut-off function $\ka_A(t)\in C^{\infty}_0(\mat{R})$ which satisfies $0\leq \ka_A \leq 1$ for all $x\in\R$, $\ka_A=1$ if $|t|\leq A$, $\ka_A=0$ if $|t|\geq A+1$ and 
 \ben
 \left|\frac{{\rm d}\ka_A}{ {\rm d}t}\right|+\left|\frac{{\rm d}^2\ka_A}{ {\rm d}t^2}\right|\leq C_1
 \enn
 for some constant $C_1$ independent of $A$. 
Moreover, we also define the domains
\ben
S^A_H:=\{x\in S_H:|x_1|\leq A  \},\qquad \G^A_H:=\{x\in\G_H:|x_1|\leq A  \},
\enn
and extends the function $\eta$ by $0$ into $\pa\Omega_{a,b}$ which is denoted by $\wid{\eta}$ and belongs to $ [L^2(\pa\Omega_{a,b})]^2$.

Using the Betti's formula, we have
\be\no
&&\int_{S_H\setminus\ov{\Om}_{a,b}}\ka_A(x_1)v\cdot\tr^*w-w\cdot\tr^*(\ka(x_1)v){\,\rm d}x\\ \no
&=&\int_{\pa S^{A+1}_H}\ka_A(x_1)v\cdot Tw-w\cdot T(\ka(x_1)v){\,\rm d}s
+\int_{\pa \Om_{a,b}}v\cdot Tw|_- -w\cdot Tv{\,\rm d}s\\\no
&=&\int_{ \G^{A+1}_H}\ka_A(x_1)v\cdot Tw-w\cdot T(\ka(x_1)v){\,\rm d}s - \int_{\G_R^{(1)}}v\cdot Tw|_- -w\cdot Tv{\,\rm d}s\\ \no
&&+\int_{\G_{a,b}}v\cdot Tw|_- -w\cdot Tv{\,\rm d}s.
\en
It is noticed by \cite{JG12} that  $v-v_0, w\in V_H$, $\tr^*v$, $\tr^*w\in [L^2(S_H\setminus\ov{\Om}_{a,b})]^2$ and $v$, $w$, $Tv$, $Tw\in[L^2(\G_H)]^2$ we have
\be\no
0&=&\int_{S_H\setminus\ov{\Om}_{a,b}}v\cdot\tr^*w-w\cdot\tr^*v{\,\rm d}x\\ \no
&=&\int_{ \G_H}v\cdot Tw-w\cdot Tv{\,\rm d}s
-\int_{\G_R^{(1)}}v\cdot Tw|_- -w\cdot Tv{\,\rm d}s
+\int_{\G_{a,b}}v\cdot Tw|_- -w\cdot Tv{\,\rm d}s
\en
by leting $A\to\infty$.

Using again the Betti's formula, it is deduced that the last integral equals $0$. Therefore, it holds 
\be\no
(L\varphi,\eta)_{[L^2(\G_{a,b})]^2}
&=& \int_{\G_{a,b}}v\cdot(Tw|_+- Tw|_-){\,\rm d}s\\ \no
&=&\left(\int_{\G_H}-\int_{\G_R^{(1)}}\right) ( v\cdot Tw-Tv\cdot w){\,\rm d}s\\\no 
&=&\int_{\G_R^{(1)}}\varphi\cdot Tw{\,\rm d}s\\\label{2.24a}
&=& \langle\varphi,T\ov{w}\rangle_{H^{1/2}\times H^{-1/2}},
\en
where we have used the fact 
\be\no
\int_{\G_H}( v\cdot Tw-Tv\cdot w){\,\rm d}s&=&\Int_{\G_H}\ov{\hat{\ov{v}}}(\xi)\cdot M(\xi)\hat{w}(\xi)-\ov{\hat{\ov{w}}}(\xi)\cdot M(\xi)\hat{v}(\xi){\,\rm d}\xi\\  \no
 &=&\int_{\G_H}\hat{v}(-\xi)\cdot M(\xi)\hat{w}(\xi)-\hat{w}(-\xi)\cdot M(\xi)\hat{v}(\xi){\,\rm d}\xi\\  \no 
 &=&\int_{\G_H}\hat{v}(-\xi)\cdot M(\xi)\hat{w}(\xi)-\hat{v}(-\xi)\cdot M(-\xi)^T\hat{w}(\xi){\,\rm d}\xi \\  \no
 &=&\int_{\G_H}\hat{v}(-\xi)\cdot (M(\xi)-M(-\xi)^T) \hat{w}(\xi){\,\rm d}\xi \\ \label{2.25a}
 &=& 0.
\en
Consequently, (\ref{2.22a}) is obtained from (\ref{2.24a}).

Let $L^*\eta=0$ for some $\eta\in [L^2(\G_{a,b})]^2$. It follows from (\ref{2.22a}) that $Tw=0$ on $\G_{R}^{(1)}$, which combines the Dirichlet boundary condition $w=0$ on $\Gamma$
with the unique continuation principle to imply that $w = 0$ in $\Om\se \G_{a,b}$. The transmission condition for $w$ on $\G_{a,b}$ now gives $\eta=0$. It means that the adjoint operator $L^*$ is injective. The proof is thus complete.
\end{proof}

\begin{lemma}\label{le2.2}
If $y\in\Om_R$, then $y\in \Om_R\se\ov{\Om}$ if and only if $G(\cdot,y,p;R)|_{\G_{a,b}}\in {\rm Range}(L)$.
\end{lemma}

\begin{proof}
If $y\in \Om_R\se\ov{\Om}$, it can be seen that $G(\cdot,y,p;R)$ is the solution to the Navi\'{e}r equation in $\Om$ satisfying 
the UPRC and the boundary condition $G(\cdot,y,p;R)=0$ on $\G^{(0)}_R$. Let $\varphi:= G(\cdot,y,p;R)|_{\G^{(1)}_R}$. Then, by the definition of $L$, one has $L\varphi=G(\cdot,y,p;R)|_{\G_{a,b}}$. Hence, $G(\cdot,y,p;R)|_{\G_{a,b}}\in {\rm Range}(L)$.

On the other hand, assume on the contrary that $y\in\Om_R$ with $y\notin \Om_R\se\ov{\Om}$. Since $G(\cdot,y,p;R)|_{\G_{a,b}}\in {\rm Range}(L)$, there exists $\psi \in [\wid{H}^{1/2}(\G_{R}^{(1)})]^2 $ such that 
$L\psi=G(\cdot,y,p;R)|_{\G_{a,b}}$. Let $v$ be the solution to Problem (\ref{2.14a}) with the boundary data $\psi$. We deduce from the analyticity of both $v$ and $G(\cdot,y,p;R)$ on $\Gamma_b$
and the unique continuation principle that $v(\cdot) = G(\cdot,y,p;R)$ in $\Omega\se\{z\}$. However, this is a contradiction for both $z\in\Om$ and $z\in\G_R^{(1)}$. This is because, if $z\in \Om$,
$v$ is continuous but $G(\cdot,y,p;R)$ is singular at $x=z$, and if $z\in \G_R^{(1)}$, $v|_{\G_R^{(1)}}\in [\wid{H}^{1/2}(\G_{R}^{(1)})]^2$ but $G(\cdot,y,p;R)|_{\G_R^{(1)}}\notin [\wid{H}^{1/2}(\G_{R}^{(1)})]^2$. The proof is thus complete.
\end{proof}

\section{The sampling method}\label{sec3}
\setcounter{equation}{0}

In this section, we concern with the inverse elastic scattering problem of recovering an unbounded rough surface by the near-field measurements. The objective is to propose a sampling-type method to image the shape and location of the unbounded rough surface. Our analysis is mainly based on the following modified near-field equation
\be\label{3.1a}
({\bf N}_{\rm Mod}g)(\cdot)=(G^{sc}(\cdot,y,p;R)+G^{in}(\cdot,y,p))|_{\G_{a,b}}\qquad {\rm for\;} \ y\in\Om_R,
\en
where $p\in\C^2$,  $G^{sc}$ is the solution of Problem (\ref{2.11a}), and ${\bf N}_{ \rm Mod}: [L^2(\G_{a,b})]^2\to [L^2(\G_{a,b})]^2 $ is defined by
\be\label{3.2a}
({\bf N}_{\rm Mod}g)(x)=\int_{\G_{a,b}}(u^{sc}(x,y,g(y);R)-G^{sc}(x,y,g(y);R)){\,\rm d}s(y) \qquad {\rm for\;} \  x\in\G_{a,b}
\en
for sufficiently large $R\in\R_+$.

It is known by (\ref{2.14a}) that ${\bf N}_{\rm Mod}g$ corresponds to the scattered field associated with the incidence operator 
${\bf H}_{\rm Mod}:[L^2(\G_{a,b})]^2\to [\wid{H}^{\frac{1}{2}}(\G_{R}^{(1)})]^2 $:
\be\label{3.3a}
({\bf H}_{\rm Mod}g)(x)=\int_{\G_{a,b}} G(x,y,g(y);R){\,\rm d}s(y) \qquad {\rm for\;} \ x \in\G_{R}^{(1)}.
\en

\begin{lemma}\label{le3.1}
If $\om^2>0$ is not a Dirichlet eigenvalue for $-\tr^*$ in $\Om_R\se\ov{\Om}$, then ${\rm Range}({\bf H}_{\rm Mod})$ is dense in $[\wid{H}^{1/2}(\G_R^{(1)})]^2$ .
\end{lemma}

\begin{proof}
First, we briefly prove the following reciprocity relation for the elastic wave scattering by a unbounded rough surface:
\be\label{3.4a}
p\cdot G(x,y,q;R)=q\cdot G(y,x,p;R) \qquad {\rm for}\ x,y\in\Om_R,\ x\neq y,\ p,q\in\C^2.
\en
Here, we remark that a similar reciprocity relation was recently proved by Liu {\em et al.} in \cite{XBH19} under a different radiation condition proposed by Arens in \cite{T02} for the pseudo stress operator
$T_{\wid{\mu},\wid{\la}}$ with $\wid{\mu}:=\mu(\mu+\la)/(3\mu+\la)$ and $\wid{\la}:=(2\mu+\la)(\mu+\la)/(3\mu+\la)$. For the sake of completeness, in the following we give a detailed proof of (\ref{3.4a})
under the radiation condition (\ref{2.7a}) and the lam{\'e} constants are given by $\wid{\mu}=0$ and $\wid{\la}=\la+\mu$ here.

To simplify the notations, let $G(x,y,q;R)$ and $G(y,z,p;R)$ be denoted by $G(x,y,q)$ and $G(y,z,p)$, respectively. For fixed $x,y\in\Om_R$ with $x\neq y$, we can first choose one constant $H'>{\rm max}\{x_2,y_2\}$ such that $x,y\in S_{R,H'}$, where $S_{R,H'}: = \Om_R\se \ov{U}_{H'}$. Define $B_\vep(x):=\{z\in\mat{R}^2:|z-x|\leq\vep \}$ and choose small enough $\vep>0$ such that $\ov{B_\vep(x)\cup B_\vep(y)}\subset S^{A+1}_{R,H'}$, where $S^A_{R,H'}:=\{x\in S_{R,H'}:|x_1|\leq A  \}$ for $A>R$.

Then, by the Betti's formula and the Dirichlet boundary conditions for $G(\cdot,x,p)$ and $G(\cdot,y,q)$, we obtain
\begin{align*}
&\Int_{S_{R,H'} \se\ov{B_\vep(x)\cup B_{\vep}(y)}}
           (\tr^* G(z,y,q)\cdot \ka_A(x_1)G(z,x,p) - G(z,y,q)\cdot \tr^* (\ka_A(x_1)G(z,x,p)) ){\,\rm d}z\\ \no
&=\Int_{S^{A+1}_{R,H'} \se\ov{B_\vep(x)\cup B_{\vep}(y)}}
           (\tr^* G(z,y,q)\cdot \ka_A(x_1)G(z,x,p) - G(z,y,q)\cdot \tr^* (\ka_A(x_1)G(z,x,p)) ){\,\rm d}z\\ \no
&= \Int_{\G^{A+1}_{H'}}+\Int_{\pa B_\vep(x)} +\Int_{\pa B_\vep(y)}
          \left(\ka_A(x_1)G(z,x,p)\cdot T_zG(z,y,q) -  T_z(\ka_A(x_1) G(z,x,p))\cdot G(z,y,q) \right)
          {\,\rm d}s(z)
\end{align*}
where $\ka_A$ is a cut-off function given in the proof of Lemma \ref{le2.1}. 
Similarily, letting $A\to\infty$ yields 
\be\no
0 &=& \Int_{S_{R,H'} \se\ov{B_\vep(x)\cup B_{\vep}(y)}}
           (\tr^* G(z,y,q;R)\cdot G(z,x,p;R) - G(z,y,q;R)\cdot \tr^* G(z,x,p;R) ){\,\rm d}z\\ \no
  &=& \Int_{\G_{H'}}+\Int_{\pa B_\vep(x)} +\Int_{\pa B_\vep(y)}
          \left(G(z,x,p;R)\cdot T_zG(z,y,q;R) -  T_z G(z,x,p;R)\cdot G(z,y,q;R) \right)
          {\,\rm d}s(z)\\ \label{3.5b}
  &=& I_1+I_2+I_3
\en
Noting $\wi{G}(\cdot,x,p;R)|_{\G_{H'} }=G(\cdot,x,p;R)_{\G_{H'}} $ and $\wi{G}(\cdot,y,q;R)|_{\G_{H'} }=G(\cdot,y,q;R)_{\G_{H'}}$,  we then have
$G(\cdot,x,p;R) $,  $G(\cdot,y,q;R) $, $TG(\cdot,x,p;R)$, $TG(\cdot,x,p;R)\in[L^2(\G_{H'}) ]^2$ from Lemma 1 of \cite{JG12} with $\wi{G}|_{S_{H'}}\in [H^1(S_{H'})]^2$, which
implies $I_1=0$ in a similar way to (\ref{2.25a}). It follows from the singularities of $G$ and $TG$ at $z=x$ and $z=y$ that $I_2=-p\cdot G(x,y,q;R)$ and $I_3=q\cdot G(y,x,p;R)$. Thus, we have proved the equality (\ref{3.4a}).

Next, for $g\in [L^2(\G_{a,b})]^2$ and $\psi\in [H^{-1/2}(\G_R^{(1)})]^2$ we use (\ref{3.4a}) to deduce 
\be\no
\langle \psi, {\bf H}_{\rm Mod}g \rangle_{H^{-1/2}\times H^{1/2}}
&=&\int_{\G_R^{(1)}}\psi(x)\cdot\ov{\int_{\G_{a,b}}G(x,y,g(y);R){\,\rm d}s(y)}{\,\rm d}s(x)\\\no
&=&\int_{\G_R^{(1)}} \int_{\G_{a,b}}\psi(x)\cdot\ov{G(x,y,g(y);R)}{\,\rm d}s(y) {\,\rm d}s(x)\\ \no
&=&\int_{\G_R^{(1)}} \int_{\G_{a,b}}\ov{g(y)}\cdot\ov{G(y,x,\ov{\psi(x)};R)}{\,\rm d}s(y) {\,\rm d}s(x)\\ \no
&=&\int_{\G_{a,b}}\ov{g(y)}\cdot\int_{\G_R^{(1)}}\ov{G(y,x,\ov{\psi(x)};R)}{\,\rm d}s(x) {\,\rm d}s(y)\\ \label{3.6a}
&=&:({\bf H}_{\rm Mod}^*\psi,g)_{[L^2(\G_{a,b})]^2\times [L^2(\G_{a,b})]^2}.
\en
It is seen from (\ref{3.6a}) that  the adjoint operator ${\bf H}_{\rm Mod}^*$ of ${\bf H}_{\rm Mod}$ is given by
\be\label{3.11a}
({\bf H}_{\rm Mod}^*\psi)(y)=\int_{\G_R^{(1)}}\ov{G(y,x,\ov{\psi(x)};R)}{\,\rm d}s(x)\qquad {\rm for}\ y\in \G_{a,b}
\en
in the sense of the $L^2$-inner product.

Assume that there exists $\psi\in [H^{-1/2}(\G_R^{(1)})]^2$ such that $ {\bf H}_{\rm Mod}^*\psi =0$. Let 
\ben
v(y)=\int_{\G_R^{(1)}}\ov{G(y,x,\ov{\psi(x)};R)}{\,\rm d}s(x)\qquad {\rm for\;\;}y\in \Om_R, 
\enn 
which gives $v=0$ on $\G_{a,b}$. It then follows from the analyticity of $v$ that $v=0$ in $\Om$. Since $G(\cdot,x,p;R)=0$ on $\wid{\G}_R$ for $x\in \G_R^{(1)}$, $p\in\C^2$, by using the jump relation for single-layered potential operator, we find that $v$ satisfies the following Dirichlet problem
\be\label{3.12a}
\tr^*v+\om^2v=0\quad {\rm in} \ \Om_R\se\ov{\Om},\qquad\ v=0\quad {\rm on} \ \pa(\Om_R\se\ov{\Om}).
\en
According to the assumption that $\om^2$ is not a Dirichlet eigenvalue for $-\tr^*$ in $\Om_R\se\ov{\Om}$, we have $v=0$ on $\Om_R\se\ov{\Om}$. Then the jump relation of the normal derivation of the single-layered potential operator gives $\psi=0$ on $\G_{a,b}$, i.e., $L^*$ is injective. This yields that  ${\rm Range}({\bf H}_{\rm Mod})$ is dense in $[\wid{H}^{1/2}(\G_R^{(1)})]^2$ . The proof is thus complete.
\end{proof}

It follows from the definitions of $L$, ${\bf H}_{\rm Mod}$ and ${\bf N}_{\rm Mod}$ that ${\bf N}_{\rm Mod}=-L{\bf H}_{\rm Mod}$. Furthermore, we have the following theorem.

\begin{theorem}\label{thm3.3}
Assume that $\om^2>0$ is not a Dirichlet eigenvalue for $-\tr^*$ in $\Om_R\se\ov{\Om}$. 

If $y\in\Om_R\se\ov{\Om}$, for $\vep>0$, there exists $g_{y,\vep}\in [L^2(\G_{a,b})]^2$ satisfying the inequality
\be\label{3.13a}
\|{\bf N}_{\rm Mod} g_{y,\vep}(\cdot)-G(\cdot,y,p;R)\|_{[L^2(\G_{a,b})]^2}\leq \vep
\en
such that $\| g_{y,\vep}\|_{[L^2(\G_{a,b})]^2}\to \infty$ and $\|v_{g_{y,\vep}}\|_{[H^1(\Om_R\se\ov{\Om})]^2}\to\infty$
as $y\to \G$, where 
\ben
v_{g_{y,\eps}}(x):=\int_{\G_{a,b}}G(x,z,g_{y,\eps}(z);R){\, \rm d}s(z) .
\enn

If $y\in\Om$, for $\vep>0$ and $\delta>0$, there exists $g^{\vep,\de}_{y,\al}\in [L^2(\G_{a,b})]^2$ satisfying the inequality
\be\label{3.14a}
\|{\bf N}_{\rm Mod} g^{\vep,\de}_{y,\al}(\cdot)-G(\cdot,y,p;R)\|_{[L^2(\G_{a,b})]^2}\leq \vep+\delta
\en
such that $\| g^{\vep,\de}_{y,\al}\|_{[L^2(\G_{a,b})]^2}\to \infty$ and $\|v_{g^{\vep,\de}_{y,\al}}\|_{[H^1(\Om_R\se\ov{\Om})]^2}\to\infty$ as $\delta\to 0$.
\end{theorem}

\begin{proof}
If $y\in \Om_R\se\ov{\Om}$, we have $G(\cdot,y,p;R)|_{\G_{a,b}}\in {\rm Range}(L)$ by Lemma \ref{le2.2}. Then there exists $\psi\in[\wid{H}^{1/2}(\G_R^{(1)})]^2$ such that $L\psi=G(\cdot,y,p;R)|_{\G_{a,b}}$. Given $\vep>0$, by Lemma \ref{le3.1}, we can choose $g_{y,\vep}\in [L^2(\G_{a,b})]^2$ such that
\be\label{3.15a}
\|{\bf H}_{\rm Mod}g_{y,\vep}-\psi\|_{[\wid{H}^{1/2}(\G_R^{(1)})]^2}\leq
\frac{\vep}{\|L\|}.
\en
Recalling the definition of $L$ from (\ref{2.21a}), we have ${\bf N}_{\rm Mod}=-L{\bf H}_{\rm Mod}$.
Consequencely, it follows 
\ben
\|{\bf N}_{\rm Mod}g_{y,\vep}(\cdot)-G(\cdot,y,p;R)\|_{[L^{2}(\G_{a,b})]^2}&=&\|L{\bf H}_{\rm Mod}g_{y,\vep}-L\psi\|_{[L^{2}(\G_{a,b})]^2}\\
&\leq& \|L\| \|{\bf H}_{\rm Mod}g_{y,\vep}-\psi       \|_{[\wid{H}^{1/2}(\G_R^{(1)})]^2} \\
&\leq&\vep.
\enn
Next, we shall show that $\| g_{y,\vep}\|_{[L^2(\G_{a,b})]^2}\to \infty$ as $y$ approaches $\G$. On the contrary,  we assume that there exists one fixed $C_0>0$ such that $\| g_{y,\vep}\|_{[L^2(\G_{a,b})]^2}\leq C_0$, as $y$ in $\Om_R\se\ov{\Om}$ tends to some point $y_0\in\G$ along the normal direction of $y_0$. Noting $\psi(\cdot)=G(\cdot,y,p;R)|_{\G}$ for $y\in \Om_R\se\ov{\Om}$, we deduce
\be\label{3.16a}
\|G^{sc}(\cdot,y,p;R)+G^{in}(\cdot,y,p) \|_{[\wid{H}^{1/2}(\G_R^{(1)})]^2}\leq
\frac{\vep}{\|L\|}+C_0\|{\bf H}_{\rm Mod}\|.
\en
from (\ref{3.15a}). However, this contradicts with the fact 
\be\no
&&\|G^{sc}(\cdot,y,p;R)+G^{in}(\cdot,y,p) \|_{[\wid{H}^{1/2}(\G_R^{(1)})]^2}\\ \no
 &&\qquad\qquad\qquad\qquad \geq  \|G^{in}(\cdot,y,p) \|_{[\wid{H}^{1/2}(\G_R^{(1)})]^2} 
       -\|G^{sc}(\cdot,y,p;R)\|_{[\wid{H}^{1/2}(\G_R^{(1)})]^2}\\ \label{3.17a}
       &&\qquad\qquad\qquad\qquad \to \ify\qquad\qquad \ {\rm as\;} \  y\to y_0
\en
due to the uniform boundedness of $\|G^{sc}(\cdot,y,p;R)\|_{[\wid{H}^{1/2}(\G_R^{(1)})]^2}$. So, we have $\| g_{y,\vep}\|_{[L^2(\G_{a,b})]^2}\to \infty$ as $y$ approaches $\G$, whence 
$\|v_{g_{y,\vep}}\|_{[H^1(\Om_R\se\ov{\Om})]^2}\to\infty$ as $y\to y_0$ follows from the trace theorem.

If $y\in\Om$, it is known  from Lemma \ref{le2.2} that $G(\cdot,y,p;R)|_{\G_{a,b}}\notin {\rm Range}(L)$. Then the the equation $L f=G(\cdot,y,p;R)|_{\G_{a,b}}$ has to be solved by considering 
its regularized equation
\be\label{3.18a}
\al f_{\al}+L^* L f_{\al} =L^* (G(\cdot,y,p;R)|_{\G_{a,b}} ) 
\en
since $L$ is an injective and compact operator with dense range, where $\al>0$ is the regularized parameter.
Using the Picard theorem, the solution to (\ref{3.18a}) can be represented as 
\be\label{3.19a}
f_\al = \sum_{n=1}^\infty\frac{\mu_n}{\al+\mu_n^2}(G(\cdot,y,p;R)|_{\G_{a,b}},\psi_n)\varphi_n
\en
where $(\mu_n,\varphi_n,\psi_n)$ is a singular system for $L$. 

By \cite[Theorem 2.13]{FD06}, for $\de>0$ we can choose an associated $\al>0$ such that
\be\label{3.20a}
\|L f^{\de}_{y,\al}(\cdot)-G(\cdot,y,p;R)|_{\G_{a,b}}\|_{[L^2(\G_{a,b})]^2}\leq \de.
\en
Furthermore, it can be concluded by the Picard theorem that $\|f^{\de}_{y,\al}\|_{[H^{1/2}(\G_R^{(1)})]^2 }\to\infty$ as $\al\to 0$, due to 
$G(\cdot,y,p;R)|_{\G_{a,b}}\notin {\rm Range}(L)$.  With the aid of Lemma \ref{le3.1}, we now arrive at that for any given $\vep>0$ 
$g^{\vep,\de}_{y,\al} \in [L^2(\G_{a,b})]^2$ could be chosen with 
\be\label{3.21a}
\|f^{\de}_{y,\al} - H_{\rm Mod}g^{\vep,\de}_{y,\al}\|_{[H^{1/2}(\G^{(1)}_R)]^2} \leq \frac{\vep}{\|L\|}.
\en
Hence, we have
\be\no
\|{\bf N}_{\rm Mod}g^{\vep,\de}_{y,\al} - G(\cdot,y,p;R)\|_{[L^2(\G_{a,b})]^2}
&=&\|L H_{\rm Mod}g^{\vep,\de}_{y,\al} - G(\cdot,y,p;R)\|_{[L^2(\G_{a,b})]^2} \\ \no
&\leq& \|L H_{\rm Mod}g^{\vep,\de}_{y,\al} - L f^{\de}_{y,\al} \|_{[L^2(\G_{a,b})]^2} \\
\no &&+ \|L f^{\de}_{y,\al} - G(\cdot,y,p;R)\|_{[L^2(\G_{a,b})]^2} \\ \label{3.22a}
&\leq& \vep+\de.
\en
Similarly, we can also deduce that both $\|g^{\vep,\de}_{y,\al}\|_{[L^2(\G_{a,b})]^2}\to\infty$ and $\|v_{g^{\vep,\de}_{y,\al}}\|_{[H^1(\Om_R\se\ov{\Om})]^2}\to\infty$ as $\de\to 0$.
The proof is thus complete.
\end{proof}

The above Theorem \ref{thm3.3} inspires us to define the following indicator function
\be\label{3.23a}
{\rm Ind}(z):=1/\|g_z\|_{[L^2(\G_{a,b})]^2 }
\en
by the solution $g_z$ of inequality (\ref{3.13a}) and (\ref{3.14a}), which has distinct behaviors for $z\in \Om_R\setminus\ov{\Om}$ and $z\in \Om$.
Based on this observation we can propose the following algorithm by ${\rm Ind}(z)$ to reconstruct the rough surface $\G$.
		
\begin{algorithm}\caption{Reconstruction of Rough Surfaces Based on Modified Near-field Equation }\label{al1}
\begin{itemize}
\item 1. Select  a sampling region $S$ containing the desired part of the rough surface $\G$.
				
\item 2. Choose a sufficiently large $R>0$ to numerically compute $G^s(x,z,p;R)$ for $x\in \G_{a,b}$.
				
\item 3.  Solve the modified near-field equation  (\ref{3.1a}) numerically to get $g_z$ for each sampling points $z\in S$.
				
\item 4. Compute the indicator function ${\rm Ind}(z)$ and choose a cut-off value $C>0$ to plot  $\Om_R\se\ov{\Om}$ if and only if ${\rm Ind}(z)\leq C$.
\end{itemize} 
\end{algorithm}
		
\begin{remark}\label{rm3.5} 
{\rm In general, it is impossible to numerically recover the entire rough surface by limited measurements in a finite subdomain above the surface. Moreover, by Theorem \ref{thm3.3}, 
it can be seen that Algorithm $1$ always works for recovering the rough surface restricted into a finite interval $[-R,R]$ for each sufficiently large $R>0$. Therefore, 
we can always choose an appropriate $R>0$ in practical applications so that the reconstructed part of the surface we want can be covered by $[-R,R]$.

}
\end{remark}
		
\section{Numerical experiments}\label{sec4}
\setcounter{equation}{0}

In this section, several numerical examples are presented  to demonstrate the effectiveness of the reduced Algorithm 1 by take elastic wave measurements on $\G_{a,b}$. Notice that Algorithm 1
does not work if the measurement data $G^{sc}(\cdot,z,p;R)|_{\G_{a,b}}$ cannot be provided for the modified near-field operator ${\bf N}_{\rm Mod}$ with a chosen $R>0$. However,
similar to the acoustical case \cite{MJKJ17}, it could be deduced that $G^{sc}(\cdot,z,p;R)$ decays to $0$ in any fixed bounded domain in $\Om$ as $R\to\infty$ under the a priori assumption 
on $\G$ which lies in an unbounded strip domain
with a finite height in $x_2$-direction. The following numerical experiments will illustrate this fact. For the sake of simplicity, we assume in numerical experiments that the rough surface $\G$ is flat 
for sufficiently large $|x_1|$, and the special surface $\wid{\G}_R$ is a local fanshaped perturbation of the plane $x_2=0.5$ with center $(0,\sqrt{3}R/2)$, center angle $2\pi/3$ and radius $R$.
We present numerical results of $G^{sc}(\cdot,z,p;R)$ to Problem \eqref{2.11a}  in Table 1 at several discretized points $x^1=(-1,1)$, $x^2=(0,1)$ and $x^3=(1,1)$ with the location 
$z=(-1,1)$ and the polarization $p=(0,1)$, where the lam{\'e} constants and frequency are chosen as $\mu=3$, $\lambda=9$ and $\om=1$ and the parameter $R=10^t$ with $t=2,3,4,5,6$.		

\begin{tablehere}
\begin{center}
\begin{tabular}{ l     l     l}
\hline
&$R$        &\qquad\;\;$G^{ sc}(x^1,z,p;R)$  \qquad\qquad\quad\; $G^{sc}(x^2,z,p;R)$ \qquad\qquad\quad\;\;$G^{sc}(x^3,z,p;R)$ \\
\hline
&$10^2$   &\begin{tabular} {l}
		1.0e-03$\cdot$(-0.5641+0.6026{\rm i})\\
		1.0e-03$\cdot$(-0.1707+0.2210{\rm i}) 
	\end{tabular} 
	\begin{tabular} {l}
		1.0e-03$\cdot$(0.2350-0.2050{\rm i})\\
	1.0e-03$\cdot$(1.6508-3.1582{\rm i}) 
	\end{tabular}\;\;\,
	\begin{tabular} {l}
		1.0e-03$\cdot$(2.0287-3.2856{\rm i})\\
		1.0e-03$\cdot$(2.1780-3.3085{\rm i}) 
	\end{tabular}\\ \hline
&$10^3$   &\begin{tabular} {l}
		1.0e-03$\cdot$(-0.0161+0.0329{\rm i}) \\
		1.0e-03$\cdot$(-0.0102+0.0094{\rm i})  
	\end{tabular} 
	\begin{tabular} {l}
		1.0e-03$\cdot$(-0.0043-0.0142{\rm i})  \\
		1.0e-03$\cdot$(0.0439-1.5410{\rm i})
	\end{tabular}\;\,
	\begin{tabular} {l}
		1.0e-03$\cdot$(0.0452-1.5421{\rm i})\\ 
		1.0e-03$\cdot$(0.0461-1.5425{\rm i}) 
	\end{tabular}\\ 
\hline
%
%
&$10^4$   &\begin{tabular} {l}
		1.0e-04$\cdot$(-0.0019-0.0014{\rm i})\\
		1.0e-04$\cdot$(-0.0012-0.0007{\rm i}) 
	\end{tabular} \,
	\begin{tabular} {l}
		1.0e-04$\cdot$(-0.0005+0.0000{\rm i})\\
		1.0e-04$\cdot$(-1.9948+0.1195{\rm i}) 
	\end{tabular}
	\begin{tabular} {l}
		1.0e-04$\cdot$(-1.9948+0.1196{\rm i})\\
		1.0e-04$\cdot$(-1.9949+0.1195{\rm i}) 
	\end{tabular}\\ 
\hline
%
%
&$10^5$   &\begin{tabular} {l}
		1.0e-05$\cdot$(-0.0002+0.0001{\rm i})\\
		1.0e-05$\cdot$(-0.0001+0.0002{\rm i}) 
	\end{tabular} 
	\begin{tabular} {l}
		1.0e-05$\cdot$(0.0000+0.0002{\rm i})\\
		1.0e-05$\cdot$(1.5585-1.2522{\rm i}) 
	\end{tabular}\;\,
	\begin{tabular} {l}
		1.0e-05$\cdot$(1.5586-1.2522{\rm i})\\
		1.0e-05$\cdot$(1.5586-1.2522{\rm i}) 
	\end{tabular}\\ 
\hline
%
%
&$10^6$   &\begin{tabular} {l}
		1.0e-07$\cdot$(-0.0001-0.0002{\rm i})\\
		1.0e-07$\cdot$(0.0001--0.0002{\rm i}) 
	\end{tabular} \,
	\begin{tabular} {l}
		1.0e-07$\cdot$(0.0002-0.0002{\rm i})\\
		1.0e-07$\cdot$(-2.3146+4.2823{\rm i}) 
	\end{tabular}
	\begin{tabular} {l}
		1.0e-07$\cdot$(-2.3146+4.2823{\rm i})\\
		1.0e-07$\cdot$(-2.3146+4.2823{\rm i}) 
	\end{tabular}\\ 
\hline
\end{tabular}
\caption{\label{us_num} Numerical solutions of $G^{ sc}(x,z,p;R)$ with different $R$.}
\end{center}
\end{tablehere}
From Table 1, it is seen that for each $J_0>0$, we can always choose sufficiently large $R>0$ such that $\|{\bf N}-{\bf N}_{\rm Mod}\|_{[L^2(\G_{a,b})]^2\to [L^2(\G_{a,b})]^2} \leq C10^{-J_0}$,
where  the operator ${\bf N}: [L^2(\G_{a,b})]^2\to [L^2(\G_{a,b})]^2 $ is defined by
\ben
({\bf N}g)(x)=\int_{\G_{a,b}}u^{sc}(x,y,g(y);R){\,\rm d}s(y) \qquad {\rm for\;} \  x\in\G_{a,b}.
\enn
Thus, the sampling method based on the modified near-field equation \eqref{3.1a} could be reduced to numerically solve the following equation
\be\label{4.1aa}
{\bf N}g_y = G^{in}(\cdot,y,p)|_{\G_{a,b}},
\en
in the sense e.g., $J_0=100$.

In view of the analyticity of the kernel of the operator ${\bf N}$, equation \eqref{4.1aa} is severally ill-posed. We then consider its regularized equation
\be\label{4.1a}
\al g_z^\al+{\bf N}^*{\bf N}g_z^\al = {\bf N}^*(G^{in}(\cdot,z,p)|_{\G_{a,b}}),
\en
where $\al>0$ is a regularization parameter chosen by the Morozov discrepancy principle. 
Therefore, the proposed Algorithm 1 can be reduced to the following algorithm 2 in numerical experiments.
\begin{algorithm}\caption{Reconstruction of Rough Surfaces Based on Near-field Equation }\label{al2}
\begin{itemize}
\item 1. Select  a sampling region $S$ containing the desired part of the rough surface $\G$.

\item 2. Compute the scattered solution $u^{sc}(x,y,p)$ for $x,y\in\G_{a,b}$.
				
\item 3.  Solve the regularized equation $\al g_z^\al+{\bf N}^*{\bf N}g_z^\al = {\bf N}^*(G^{in}(\cdot,z,p)|_{\G_{a,b}}$ to get $g_z$ for each sampling points $z\in S$.
				
\item 4. Plot the indicator function ${\rm Ind}(z) = 1/\|g_z\|_{L^2}$ to image the domain $\Om_R\se\ov{\Om}$ by choosing a cut-off value $C>0$.
\end{itemize} 
\end{algorithm}

In numerical experiments, the exact measurement data $u^{sc}$ is obtained by the Nystr\"{o}m method (cf.\cite{ATSNA00}). Let  $x_k$, $k=1,2,\cdots,N$, denote $N$ measuring points on $\G_{a,b}$ and $x_{sl}$, $l=1,2,\cdots,N$, denote $N$ incident point sources located at 
$\G_{a,b}$. Then the near-field operator ${\bf N}$ is discretized into a $(2N\ti 2N)$-dimensional matrix 
\be\label{4.2a}
({\bf N})_{2N\ti 2N} = (u^{sc}(x_k,x_{sl}) )_{1\leq k,l\leq N}
\en
with $u^{sc}(x_k,x_{sl})\in\mat{C}^{2\ti 2}$ defined by $u^{sc}(x_k,x_{sl})e_j:=u^{sc}(x_k,x_{sl},e_j)$, and the incident source $G^{in}$ is discretized into a $2N$-dimensional vector
\be\label{4.3a}
(G^{in}(x_k,z,p;R))_{1\leq k\leq N}.
\en
Thus, we can define the following discretized form 
\be\label{4.4a}
{\rm Ind}_N:\;z\mapsto 1/\|\hat{g}_z\|_{\ell^2}: = 1\left/\left[\sum_{j=1}^{2N}|\hat{g}_{z,j}|^2\right]^{\frac{1}{2}}\right.
\en
of the indicator function ${\rm Ind}(z)$ by solving numerically the equation \eqref{4.1a}, 
which behaves differently inside and outside the domain $\Om_R\setminus\ov{\Om}$. To deal with different examples in the same manner, we introduce another new indicator function		
\be\label{4.5a}
{\rm NInd}(z): ={\rm Ind}_N(z)/\max\limits_{z\in S} {\rm Ind}_N(z)
\en
which will be used to image the rough surface in all numerical examples. Moreover, we also examine our sampling algorithm with the different noisy level, where 
the noisy data is considered to be a perturbation of the exact data in the form 
\be\label{4.6a}
(({\bf N})_{2N\times 2N})_\delta: = ({\bf N})_{2N\times 2N} + \delta \frac{\bf X}{\|{\bf X}\|_{2}} \|({\bf N})_{2N\times 2N} \|_{2}.
\en
Here, ${\bf X}$ is a complex-valued noisy matrix containing random numbers that are uniformly distributed in the complex square $\{c_1+{\rm i}c_2: |c_1|\leq 1, \;|c_2|\leq 1 \}\subset \C$. 
In this case, the corresponding indicator function is denoted by $({\rm NInd})_{\de}(z)$.		

In the following numerical examples, we set the grid $S=(-5,-5)\ti(0,0.9)$, and the step size in $x_1$-axis and $x_2$-axis are $0.1$ and $0.09$, respectively. Besides, we set the lam{\'e} constants $\mu=3$ and $\la=9$, the frequency $\om=1$, the measurement width $a=10$, the measurement height $b=1$ and the number of the measurement points $N=401$ which are located equidistantly in $\G_{a,b}$. 
Moreover, numerous numerical examples we carry out indicate that it is enough for us to take the regularization parameter $\al$ as a fixed constant. Here, we choose $\al$ to be $10^{-3}$ in (\ref{4.1a}). 
		
{\bf Example 1.} In this example, the rough surface $\G$ is given by  
\ben
f(x_1)=
\begin{cases}
	0.5+0.8e^{\frac{16 }{x_1^2-16}}\quad &{\rm for }\; |x_1|<4,\\
	0.5                           \quad &{\rm for }\; |x_1|\geq 4
\end{cases}
\enn 
which is a local perturbation of the plane surface $x_2=0.5$ with a small rate of change. The reconstructions are presented in Figure 2 with no noise $(a)$, $2\%$ noise $(b)$ and
$5\%$ noise, respectively.


\begin{figurehere}
\vskip 0.5truecm
\hfill{}\includegraphics[clip,width=0.31\textwidth]{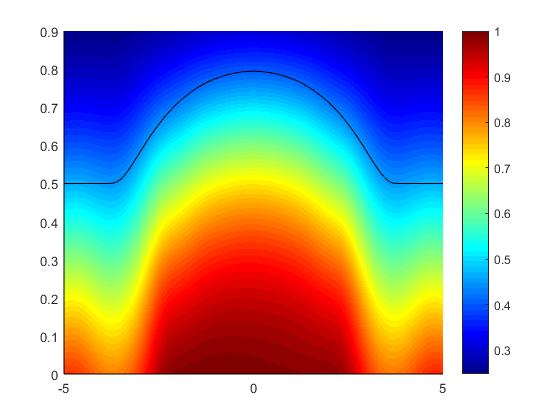}\hfill{}
	\hfill{}\includegraphics[clip,width=0.31\textwidth]{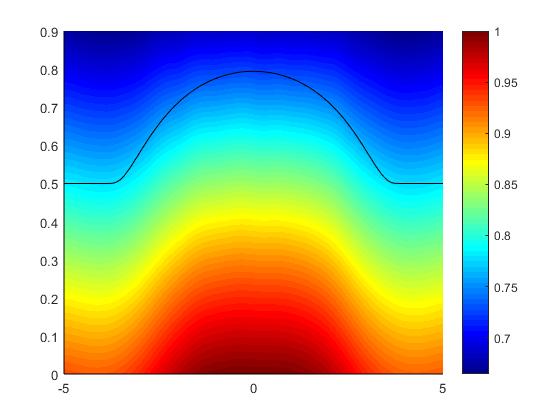}\hfill{}
	\hfill{}\includegraphics[clip,width=0.31\textwidth]{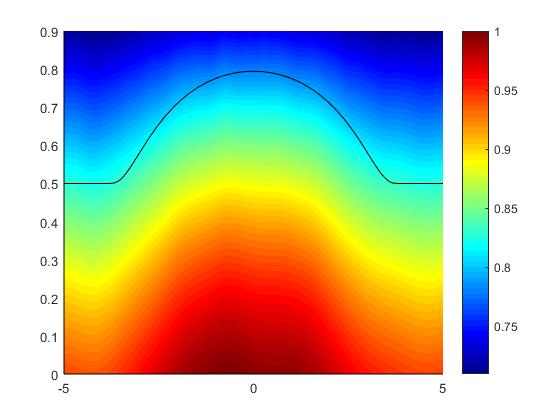}\hfill{}
	
\hfill{} (a) No noise \hfill{}\hfill{}(b) 2$\%$ noise  \hfill{}\hfill{}(c) 5$\%$ noise\hfill{}
\vskip 0.5truecm
	
\caption{\label{Figure2}  \small{Reconstruction of the local rough surface $\G$ given in Example 1 from data with  no noise (a), 2\% noise (b) and 5\% noise by the indicator function (\ref{4.5a}).}}
\end{figurehere}

{\bf Example 2.} In this example, the rough surface $\G$ is described by $f(x_1)=0.5+\Om_3(x_1)$, 
where $\Om_3(\cdot)$ is a cubic B-spline function which is twice differentiable with compactly support in $\mat{R}$ and is given by 
\ben
\Om_3(x_1)=
\begin{cases}
	\frac{1}{2}|x_1|^3-x_1^2+\frac{2}{3}                  \quad &{\rm for }\;|x_1|\leq 1,\\
	-\frac{1}{6}|x_1|^3+x_1^2-2|x_1|+\frac{4}{3 }  \quad &{\rm for }\; 1<|x_1|\leq 2,\\
	0                                          \quad &{\rm for }\; |x_1|\geq 2.\\
\end{cases}
\enn
Compared to Example 1, $\G$ is considered in this case to remain a local perturbation of the plane surface $x_2=0.5$, but with a large rate of change; see the reconstructions
in Figure 3 with no noise $(a)$, $2\%$ noise $(b)$ and $5\%$ noise, respectively.

\begin{figurehere}
\vskip 0.5truecm
\hfill{}\includegraphics[clip,width=0.31\textwidth]{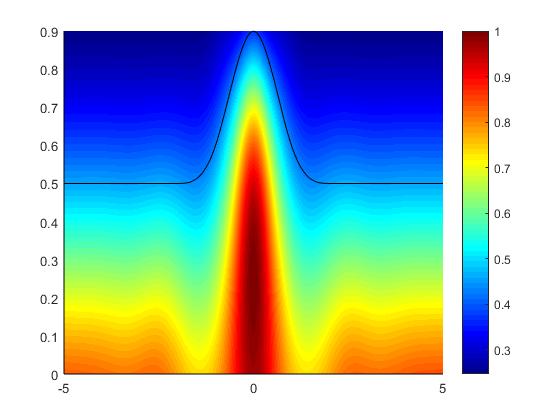}\hfill{}
	\hfill{}\includegraphics[clip,width=0.31\textwidth]{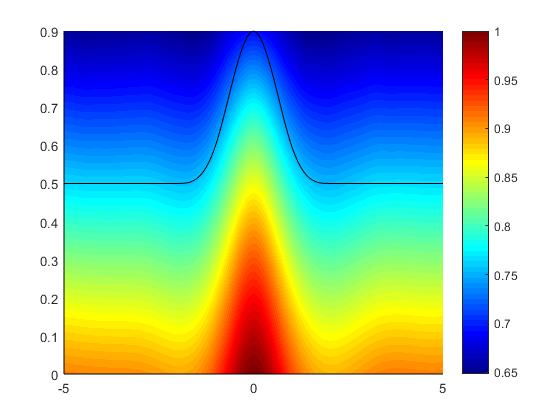}\hfill{}
	\hfill{}\includegraphics[clip,width=0.31\textwidth]{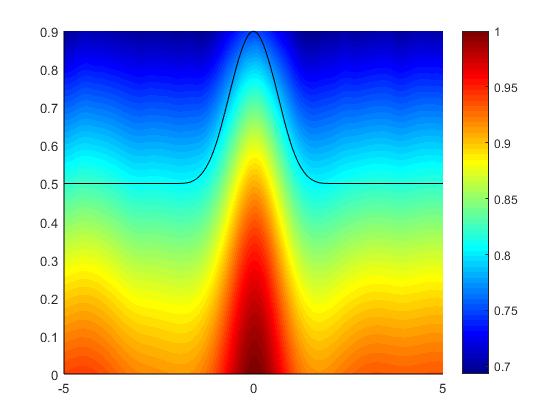}\hfill{}

\hfill{} (a) No noise \hfill{}\hfill{}(b) 2$\%$ noise  \hfill{}\hfill{}(c) 5$\%$ noise\hfill{}
\vskip 0.5truecm
	
\caption{\label{Figure3}  \small{Reconstruction of the local rough surface $\G$ given in Example 2 from data with no noise (a), 2\% noise (b) and 5\% noise (c) by the indicator function ( \ref{4.5a}).}}
\end{figurehere}

{\bf Example 3.} In this example, $\G$ is considered to be a nonlocal rough surface described by a periodic function in the $x_1$-direction: $f(x_1)=0.5+0.15\sin(x_1)$ for $x_1\in\R$.
Numerical results are presented in Figure 4 with no noise $(a)$, $2\%$ noise $(b)$ and $5\%$ noise, respectively.
\begin{figurehere}
\vskip 0.5truecm
\hfill{}\includegraphics[clip,width=0.31\textwidth]{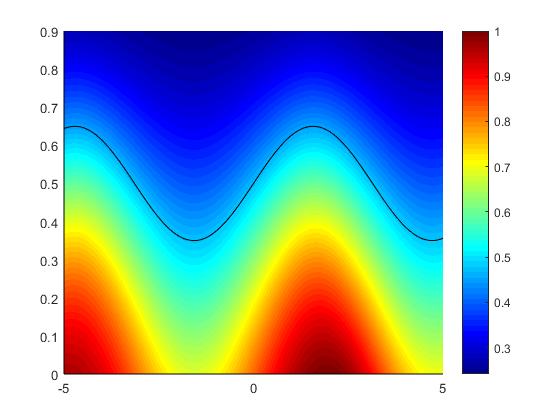}\hfill{}
	\hfill{}\includegraphics[clip,width=0.31\textwidth]{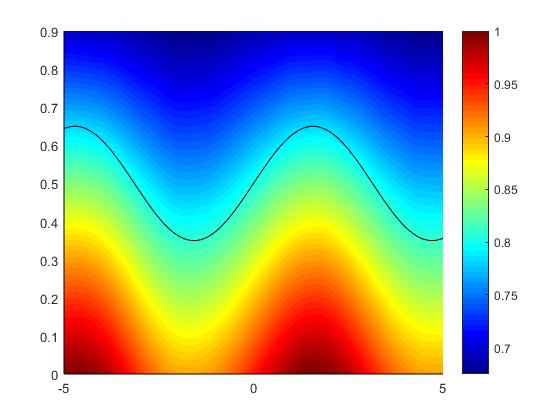}\hfill{}
	\hfill{}\includegraphics[clip,width=0.31\textwidth]{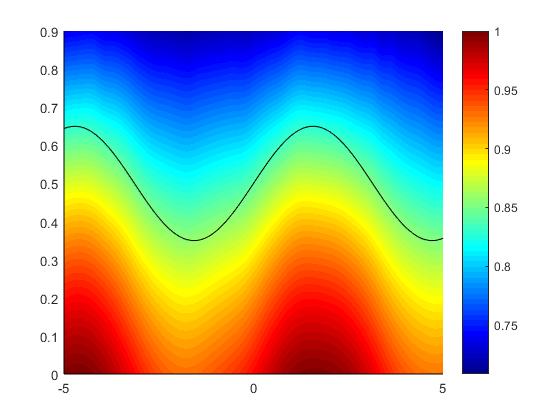}\hfill{}
	
\hfill{} (a) No noise \hfill{}\hfill{}(b) 2$\%$ noise  \hfill{}\hfill{}(c) 5$\%$ noise\hfill{}
\vskip 0.5truecm
	
\caption{\label{Figure4}  \small{Reconstruction of the periodic surface $\G$ given in Example 3 from data with no noise (a), 2\% noise (b) and 5\% noise (c) by the indicator function (\ref{4.5a}).}}
\end{figurehere}

{\bf Example 4.} In this example,  $\G$ is a nonlocal rough surface described by $f(x_1)=0.5+0.18\sin(x_1)+0.15\cos (\frac{x_1}{2})$ for $x_1\in\R$. Compared to Example 3, $\G$ 
is considered in this case to be a periodic surface in the $x_1$-direction any more. Numerical results are presented in Figure 5 with no noise $(a)$, $2\%$ noise $(b)$ and $5\%$ noise, respectively.
\begin{figurehere}
	\vskip 0.5truecm
	\hfill{}\includegraphics[clip,width=0.31\textwidth]{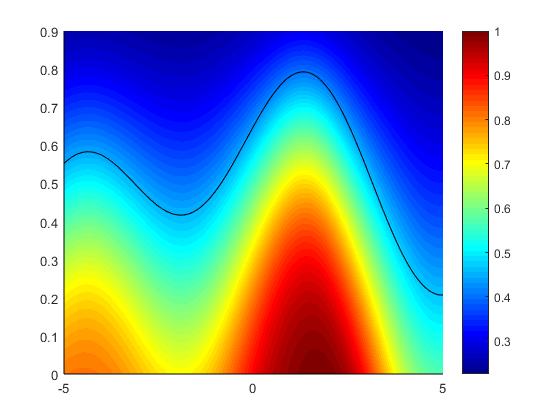}\hfill{}
		\hfill{}\includegraphics[clip,width=0.31\textwidth]{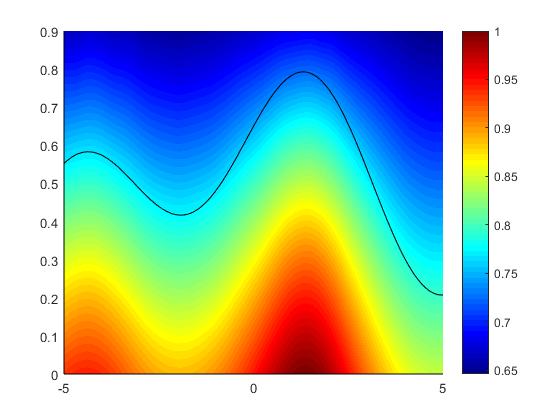}\hfill{}
		\hfill{}\includegraphics[clip,width=0.31\textwidth]{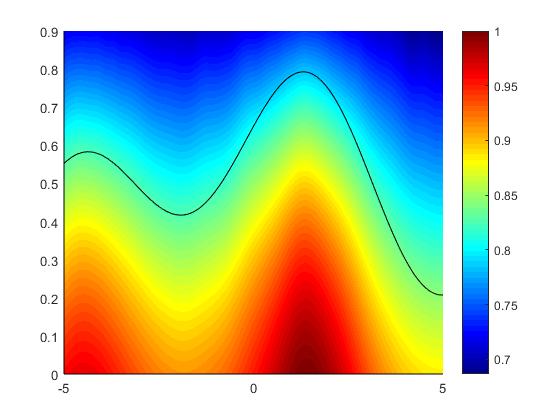}\hfill{}
		
\hfill{} (a) No noise \hfill{}\hfill{}(b) 2$\%$ noise   \hfill{}\hfill{}(c) 5$\%$ noise\hfill{}
	\vskip 0.5truecm
	
	\caption{\label{Figure5}  \small{Reconstruction of the nonlocal rough surface $\G$ given in Example 4 from data with no noise (a), 2\% noise  (b) and 5\% noise (c) by the indicator function (\ref{4.5a}).}}
\end{figurehere}

From the above numerical examples Figure 2-5, it is observed that the sampling method proposed in Theorem \ref{thm3.3} can give a satisfactory construction for the rough surface $\G$ with various 
geometric features at different noise levels. Moreover, it is also observed that the quality of the reconstruction for $\G$ depends on the amount of measurement data, which shows that 
$\G$ can be reconstructed in numerics more fully as long as more enough data are measured on the line $x_2=b$.


\section{Conclusions}

In this paper, we proposed an extended sampling method to recover an infinite rough surface by near-field elastic measurements. The idea is mainly based on constructing a modified near-field equation 
by transforming the original scattering problem into an equivalent boundary value problem with the boundary data of compact support. Numerical results demonstrate that 
the reduced inversion algorithm can work well for imaging a variety of rough surfaces to capture the essential features using both exact and noise data. Moreover, it is also shown that the more rough surface can be reconstructed if the measurement width $a$ becomes larger.
Further, the reconstruction can also be regarded as a good initial guess for an iterative type method in order to obtain an accurate numerical reconstruction of the surface.


\section*{Acknowledgements}

The work is partially supported by the NNSF of China Grant No. 11771349. The authors thank Dr. Jianliang Li for his discussion for numerically solving the elastic wave scattering problem 


\end{document}